\documentclass[11pt,a4paper,leqno]{amsart}
\usepackage{amsfonts}
\usepackage{amsfonts}
\usepackage{amsfonts}
\usepackage{amsfonts}
\usepackage{amsfonts}
\usepackage{amsfonts}
\usepackage{amsfonts}
\usepackage{mathrsfs}
\usepackage{mathrsfs}
\usepackage{mathrsfs}
\usepackage{mathrsfs}
\usepackage{mathrsfs}
\usepackage{mathrsfs}
\usepackage{mathrsfs}
\usepackage{mathrsfs}
\usepackage{mathrsfs}
\usepackage{mathrsfs}
\usepackage{mathrsfs}
\usepackage{mathrsfs}
\usepackage{amsthm}
\usepackage{xspace}
\usepackage{graphics}
\usepackage{amsfonts,amssymb}
\usepackage{amsthm}
\usepackage[all]{xy}
\usepackage{stmaryrd}
\usepackage{color}

\newcommand{\be}{\begin{otherlanguage}{english}}
\newcommand{\ee}{\end{otherlanguage}}

\theoremstyle{definition}

\newtheorem{rem}{Remark}
\theoremstyle{plain}
\newtheorem{lem}{Lemma}
\newtheorem{prop}[lem]{Proposition}
\newtheorem{thm}[lem]{Theorem}
\newtheorem{cor}[lem]{Corollary}

\newtheorem*{thm*}{Theorem}

\theoremstyle{remark}

\numberwithin{equation}{subsection}

\newcommand{\beq}{\begin{equation}}
\newcommand{\eeq}{\end{equation}}

\begin{document}
\title{A generalization of the line translation theorem}

\author{Jian Wang}
\address{Department of Mathematics, Tsinghua University, Beijing 100084,
P.R.China.}\email{wjian05@mails.tsinghua.edu.cn}
\curraddr{\textsc{LAGA UMR 7539 CNRS, Universit\'e Paris 13, 93430
Villetaneuse, France.}} \email{wangjian@math.univ-paris13.fr}
\thanks{The authors have been supported by the project
111-2-01.}

\date{April 20, 2011}
\maketitle

\begin{abstract}Through the method of brick decomposition and
the operations on  essential topological lines, we generalize the
line translation theorem of Beguin, Crovisier, Le Roux \cite{B2} in
the case where the property of preserving a finite measure with
total support is replaced by the intersection property.
\end{abstract}\bigskip

\section{introduction}
Let $\mathbb{A}=\mathbb{R}/\mathbb{Z}\times \mathbb{R}$ be the open
annulus. We denote  $\pi$ the covering map
\begin{eqnarray*}
\pi\,:\, \mathbb{R}^2&\rightarrow& \mathbb{A} \\
(x,y)&\mapsto&(x+\mathbb{Z},y),
\end{eqnarray*}
and  $T$ the generator of the covering transformation group
\begin{eqnarray*}
T\,:\, \mathbb{R}^2&\rightarrow& \mathbb{R}^2 \\
(x,y)&\mapsto&(x+1,y).
\end{eqnarray*}
Write respectively $S$ and $N$ for the lower and the upper end of
$\mathbb{A}$. We call \emph{essential line} in $\mathbb{A}$ every
simple path, parametrized by $\mathbb{R}$, properly embedded in
$\mathbb{A}$, joining one end to the other one. We call
\emph{essential circle} in $\mathbb{A}$ every simple closed curve
which is not null-homotopic.

Let $f$ be a homeomorphism of $\mathbb{A}$. We say that $f$
satisfies the \emph{intersection property} if  any essential circle
in $\mathbb{A}$ meets its image by $f$. We denote the space of all
homeomorphisms of $\mathbb{A}$ which are isotopic to the identity as
$\mathrm{Homeo}_*(\mathbb{A})$ and its subspace whose elements
additionally have the intersection property as
$\mathrm{Homeo}_*^\wedge(\mathbb{A})$. If $X$ is a topological space
and $A$ is a subset of $X$, denote respectively by
$\mathrm{Int}(A)$, $\mathrm{Cl}(A)$ and $\partial A$ the interior,
the closure and the boundary of $A$.

The goal of the paper is to generalize the line translation theorem
of B\'{e}guin, Crovisier, Le Roux \cite{B2} in the case where the
property of preserving a finite measure with total support is
replaced by the intersection property. A similar result has done
when $\mathbb{A}$ is a closed annulus due to B\'{e}guin, Crovisier,
Le Roux and Patou \cite{B1}. 
The reason why we consider the intersection property is
that some interesting questions can be reduced to homeomorphisms of
the open annulus which satisfy the intersection property but do not
preserve a finite measure. For example, consider a homeomorphism $f$
of a closed surface $M$ with genus at least one that preserves a
finite measure with total support. Take a lift $F$ to the universal
covering space $\widetilde{M}$ (homeomorphic to the Poincar\'{e}
disk) of $M$ and suppose that $F$ has a fixed point. When we remove
this point, we obtain a map of the open annulus that satisfies the
intersection property but does not preserve a finite measure. We
will strongly use the arguments of Beguin, Crovisier, Le Roux
\cite{B2} but will have to add some crucial lemmas to weaken their
theorem. Let us first recall some results and then state our main
results.

When $f\in\mathrm{Homeo}_*(\mathbb{A})$, we define the rotation
number of a positively recurrent point as follows (see \cite{P2} for
details). We say that a positively recurrent point $z$ has a
\emph{rotation number} $\rho(F;z)\in \mathbb{R}$ for a lift $F$ of
$f$ to the universal covering space $\mathbb{R}^2$ of $\mathbb{A}$,
if for every subsequence $\{f^{n_k}(z)\}_{k\geq 0}$ of
$\{f^n(z)\}_{n\geq 0}$ which converges to $z$, we have
\[\lim_{k\rightarrow+\infty}\frac{p_1\circ
F^{n_k}(\widetilde{z})-p_1(\widetilde{z})}{n_k}=\rho(F;z)\] where
$\widetilde{z}\in \pi^{-1}(z)$ and $p_1$ is the first projection
$p_1(x,y)=x$. In particular, the rotation number $\rho(F;z)$ always
exists and is rational when $z$ is a fixed or periodic point of $f$.
Let $\mathrm{Rec}^+(f)$ be the set of positively recurrent points of
$f$. We denote the set of rotation numbers of positively recurrent
points of $f$ as $\mathrm{Rot}(F)$.

It is well known that a positively recurrent point of $f$ is also a
positively recurrent point of $f^q$ for all $q\in \mathbb{N}$ (we
give a proof in the Appendix, see Lemma \ref{subsec:positively
recurrent}). By the definition of rotation number, we easily get
that $\mathrm{Rot}(F)$ satisfies the following elementary
properties.
\begin{enumerate}\label{prop:ROT}
  \item[1.] $\mathrm{Rot}(T^k\circ F)=\mathrm{Rot}(F)+k$
  for every $k\in \mathbb{Z}$;
  \item[2.] $\mathrm{Rot}(F^q)=q\mathrm{Rot}(F)$ for every $q\in \mathbb{N}$.
\end{enumerate}

We recall that a \emph{Farey interval} is an interval of the form $]
\frac{p}{q},\frac{p'}{q'} [$ with $q,q'\in
\mathbb{N}\setminus\{0\}$, $p,p'\in \mathbb{Z}$ and $qp'-pq'=1$. Our
main result is the following:
\begin{thm}[Generalization of the line translation theorem]\label{thm:the line translation theorem}
Let $f\in\mathrm{Homeo}_*^\wedge(\mathbb{A})$ and $F$ be a lift of
$f$ to $\mathbb{R}^2$. Assume that $\mathrm{Rot}(F)\neq \emptyset$
and its closure is contained in a Farey interval
$]\frac{p}{q},\frac{p'}{q'}[$. Then, there exists an essential line
$\gamma$ in $\mathbb{A}$ such that the lines
$\gamma,f(\gamma),\cdots, f^{q+q'-1}(\gamma)$ are pairwise disjoint.
Moreover, the cyclic order of these lines is the same as the cyclic
order of the $q+q'-1$ first iterates of a vertical line
$\{\theta\}\times\mathbb{R}$ under the rigid rotation with angle
$\rho$, for any $\rho\in]\frac{p}{q},\frac{p'}{q'}[$.
\end{thm}
It is easy to see that a homeomorphism of $\mathbb{A}$ that
preserves a finite measure with total support satisfies the
intersection property. Note that the statement of Theorem
\ref{thm:the line translation theorem} with this stronger condition
is exactly what is done in \cite{B2}.
\bigskip

Let us consider the simple case of Theorem \ref{thm:the line
translation theorem} when $(p,q)=(0,1)$ and $(p',q')=(1,1)$. We
obtain the following corollary.

\begin{cor}\label{thm:(0,1)}Let $f\in \mathrm{Homeo}_*^\wedge(\mathbb{A})$.
We suppose that $F$ is a lift of $f$ to $\mathbb{R}^2$ and that
$$\emptyset\neq \mathrm{Cl}(\mathrm{Rot}(F))\subset]0,1[.$$ Then there exists
an essential line in $\mathbb{A}$ that is free for $f$, that means
disjoint from its image by $f$.
\end{cor}

A similar result in the case when $\mathbb{A}$ is a closed annulus
is known due to Bonatti and Guillou \cite{G}: if $f$ is a
homeomorphism of the closed annulus
$\mathbb{R}/\mathbb{Z}\times[0,1]$ isotopic to the identity and
fixed point free, then either there is a free simple path for $f$
that joins the two boundary of the closed annulus or there is a free
essential circle in the closed annulus for $f$.

As we will recall later, the rotation number set of $F$ is a closed
interval if $f$ satisfies the intersection property and $F$ is any
lift of $f$ to the universal covering space $\mathbb{R}\times[0,1]$.
Therefore, if the map $f$ has no fixed point,  we can find a lift
$F$ of $f$ to $\mathbb{R}/\mathbb{Z}\times[0,1]$ such that the
rotation number set of $F$ is contained in $]0,1[$.
\smallskip

As an immediate corollary of Theorem \ref{thm:the line translation
theorem}, we get the following generalization of the Corollary 0.3
in \cite{B2}:
\begin{cor}
We suppose that $f\in \mathrm{Homeo}_*^\wedge(\mathbb{A})$ has a
rotation set reduced to a single irrational number $\rho$ (for any
given lift $F$ of $f$). Then, for every
$n\in\mathbb{N}\setminus\{0\}$, there exists an essential line
$\gamma$ in $\mathbb{A}$, such that the lines $\gamma,
f(\gamma),\cdots, f^n(\gamma)$ are pairwise disjoint. The cyclic
order of the lines $\gamma,f(\gamma),\cdots, f^n(\gamma)$ is the
same as the cyclic order of the $n$ first iterates of a vertical
line under the rigid rotation of angle $\rho$.
\end{cor}
Theorem \ref{thm:the line translation theorem} above is a global
result on the open annulus. A local version was studied by
 Patou in her thesis \cite{P} under similar hypotheses of the line translation theorem in
\cite{B2}. More precisely, she considered a local homeomorphism $F$
between two neighborhoods of the origin $O$ in $\mathbb{R}^2$ that
fixes $O$. If $F$ preserves the orientation and the area of
$\mathbb{R}^2$, and  it has no other periodic point except $O$ in
its domain, then for every $N\in \mathbb{N}$, there exists a simple
arc $\Gamma$ in a small neighborhood of $O$ that issues from $O$
such that the arcs $\Gamma,F(\Gamma),\cdots,F^N(\Gamma)$ are
pairwise disjoint except at $O$.
\smallskip

We will introduce some mathematical objects and recall some
well-known facts in Section 2. In Section 3, we first state some
crucial lemmas without proof and then prove Theorem \ref{thm:the
line translation theorem}. In Section 4, we prove the lemmas stated
in Section 3, and we give some remarks about the relations between
the positively recurrent set and the rotation number set. In Section
5, we define a weak rotation number which is a generalization of the
rotation number we have defined above. We prove the generalization
of the line translation theorem in the weak sense. Finally, in
Section 6, we provide a proof of a well known fact that we require
in this paper but were unable to find  in the literature.
\smallskip

\noindent\textbf{Acknowledgements.} I would like to thank Patrice Le
Calvez for many helps for this paper.

\section{Preliminaries}
\subsection{Essential topological lines in
$\mathbb{R}^2$} A \emph{topological line} in $\mathbb{R}^2$ is the
image of a proper continuous embedding of $\mathbb{R}$.
Equivalently, using Schoenflies theorem, it is the image of an
Euclidean line under a homeomorphism of $\mathbb{R}^2$. Let $\Gamma$
be a topological line whose orientation is induced by a
parametrization. We denote by $L(\Gamma)$ the connected component of
$\mathbb{R}^2\setminus\Gamma$ on the left of $\Gamma$, and by
$R(\Gamma)$ the connected component of $\mathbb{R}^2\setminus\Gamma$
on the right of $\Gamma$. We get a partial order relation $\leq$ on
the set of oriented lines by writing: $\Gamma_1\leq \Gamma_2$ if
$L(\Gamma_1)\subset L(\Gamma_2)$ (or equivalently
$R(\Gamma_2)\subset R(\Gamma_1)$). We get also a transitive relation
$<$ by writing $\Gamma_1<\Gamma_2$ if
$\mathrm{Cl}(L(\Gamma_1))\subset L(\Gamma_2)$ (or equivalently
$\mathrm{Cl}(R(\Gamma_2))\subset R(\Gamma_1)$).

We call \emph{essential line} in $\mathbb{R}^2$ an oriented line
$\Gamma$ such that
$\lim\limits_{t\rightarrow+\infty}p_2(\Gamma(t))=+\infty$ and
$\lim\limits_{t\rightarrow-\infty}p_2(\Gamma(t))=-\infty$ where
$p_2(x,y)=y$ and $t\mapsto\Gamma(t)$ is a parametrization. If
$\Gamma$ is a line and $F$ is a homeomorphism of $\mathbb{R}^2$,
then $F(\Gamma)$ is a line. Moreover, if $\Gamma$ is oriented, then
there is a natural orientation for $F(\Gamma)$. We say that an
orientation preserving homeomorphism $F$ of $\mathbb{R}^2$ is
\emph{essential} if for every essential line $\Gamma$ in
$\mathbb{R}^2$, $F(\Gamma)$ is also an essential line in
$\mathbb{R}^2$. Any homeomorphism $F$ that lifts a homeomorphism $f$
of an open annulus $\mathbb{A}$ isotopic to the identity is an
essential homeomorphism.

Let $\Gamma_1$ and $\Gamma_2$ be two essential lines in
$\mathbb{R}^2$, and let $U$ be the unique connected component of the
set $L(\Gamma_1)\cap L(\Gamma_2)$ which contains half lines of the
form $]-\infty,a[\times\{b\}$ for some numbers $a$ and $b$. Then the
boundary of $U$ is an essential line in $\mathbb{R}^2$, denoted by
$\Gamma_1\vee\Gamma_2$ (the proof of this fact uses a classical
result by B. Ker\'{e}kj\'{a}rt\'{o} \cite{BK}. See \cite{B1} and
\cite{B2}).

\begin{rem}\label{rem:v}Let $\Gamma_1,\Gamma_2,\Gamma_3$ be three essential
lines in $\mathbb{R}^2$. The following properties are immediate
consequences of the definition of the line $\Gamma_1\vee\Gamma_2$.
\begin{enumerate}
                          \item The line $\Gamma_1\vee\Gamma_2$ is
                          included in the union of the lines
                          $\Gamma_1$ and $\Gamma_2$. Hence, if
                          $\Gamma_3<\Gamma_1$ and
                          $\Gamma_3<\Gamma_2$, then
                          $\Gamma_3<\Gamma_1\vee\Gamma_2$.
                          \item The sets $R(\Gamma_1)$ and
                          $R(\Gamma_2)$ are included in the set
                          $R(\Gamma_1\vee\Gamma_2)$. In other words,
                          we have $\Gamma_1\vee\Gamma_2\leq\Gamma_1$
                          and $\Gamma_1\vee\Gamma_2\leq\Gamma_2$.
                        \end{enumerate}
\end{rem}

\subsection{Brouwer theory}\label{subsec:Brouwer theory}A
\emph{Brouwer homeomorphism} is an orientation preserving fixed
point free homeomorphism of $\mathbb{R}^2$. Given a Brouwer
homeomorphism $F$, a \emph{Brouwer line} for $F$ is a topological
line $\Gamma$, disjoint from $F(\Gamma)$ and separating $F(\Gamma)$
from $F^{-1}(\Gamma)$. The Brouwer Plane Translation Theorem asserts
that every point belongs to a Brouwer line.

Now we introduce the following ``equivariant'' version of the
Brouwer Plane Translation Theorem which has been proved by Guillou
and Sauzet:

\begin{thm}[\cite{GL},\cite{A}]\label{thm:GS}Let $f\in\mathrm{Homeo}_*(\mathbb{A})$
 and $F$ be a lift of
$f$ to $\mathbb{R}^2$. If $F$ is fixed point free, then there exists
an essential circle in $\mathbb{A}$ that is free under $f$ (and
therefore that lifts to a Brouwer line of $F$) or there exists an
essential line in $\mathbb{A}$ that lifts to a Brouwer line of $F$.
\end{thm}

Let $\Gamma$ be a Brouwer line. It can be oriented such that
$\Gamma<F(\Gamma)$. Since $F$ preserves the orientation, we have
$F^k(\Gamma)<F^{k+1}(\Gamma)$. By induction, we see that
$F^p(\Gamma)<F^q(\Gamma)$ if and only if $p<q$. In particular, the
lines $(F^k(\Gamma))_{k\in \mathbb{Z}}$ are pairwise disjoint.

Now let $U$ be the open region of $\mathbb{R}^2$ situated between
the lines $\Gamma$ and $F(\Gamma)$, and $\mathrm{Cl}(U)=\Gamma\cup
U\cup F(\Gamma)$. The sets $(F^k(U))_{k\in \mathbb{Z}}$ are pairwise
disjoint. As a consequence, the restriction of $F$ to the open set
$O_U=\bigcup_{k\in \mathbb{Z}} F^k(\mathrm{Cl}(U))$ is conjugate to
a translation. In particular, if the iterates of $\mathrm{Cl}(U)$
cover the whole plane, then $F$ itself is conjugate to a
translation.

\subsection{Franks' Lemma and brick decompositions}
A \emph{free disk chain} for a homeomorphism $F$ of $\mathbb{R}^2$
is a finite set $b_i\,\,(i=1,2,\cdots,n)$ of embedded open disks in
$\mathbb{R}^2$ satisfying
\begin{enumerate}
                            \item $F(b_i)\cap
                            b_i=\emptyset$ for $1\leq i\leq
                            n$;
                            \item if $i\neq j$ then either $b_i=b_j$
                            or $b_i\cap b_j=\emptyset$;
                            \item for $1\leq i\leq
                            n$, there exists $m_i>0$ such that $F^{m_i}(b_i)\cap
                            b_{i+1}\neq\emptyset$.
                          \end{enumerate}
 We say that $(b_i)_{i=1}^n$ is a \emph{periodic free
disk chain} if $b_1=b_n$.
\smallskip

In \cite{F}, Franks got the following useful lemma about the
existence of fixed points of an orientation preserving homeomorphism
$F$ of $\mathbb{R}^2$ from Brouwer theory.

\begin{prop}[Franks' Lemma]\label{prop:Franks' Lemma}
Let $F: \mathbb{R}^2\rightarrow \mathbb{R}^2$ be an orientation
preserving homeomorphism which possesses a periodic free disk chain.
Then $F$ has at least one fixed point.
\end{prop}

A \emph{brick decomposition} of a surface $S$ (not necessarily
closed) is given by a one dimensional stratified set $\Sigma$ (the
\emph{skeleton} of the decomposition) with a zero dimensional
submanifold $V$ such that any vertex $v\in V$ is local the extremity
of exactly three edges. A brick is the closure of a connected
component of $S\setminus \Sigma$. Write $\mathfrak{B}$ for the set
of bricks. For any $\mathfrak{X}\subset \mathfrak{B}$ the union of
bricks which are in $\mathfrak{X}$ is a sub-surface of $S$ with
boundary. Suppose that $F$ is a homeomorphism of $S$. We can define
the relation $B\mathscr{R}B'\Leftrightarrow F(B)\cap
B'\neq\emptyset$ where $B,B'\in \mathfrak{B}$, and write $B\preceq
B'$ if there exists a sequence $(B_i)_{0\leq i\leq n}\subset
\mathfrak{B}$ such that $B_0=B$, $B_n=B'$ and
$B_i\mathscr{R}B_{i+1}$ for every $i\in\{0,\ldots,n-1\}$. The union
$B_{\succeq}=\bigcup_{B'\succeq B}B'$ (resp.
$B_{\preceq}=\bigcup_{B'\preceq B}B'$) is a closed subset satisfying
$F(B_{\succeq})\subset \mathrm{Int}(B_{\succeq})$ (resp.
$F^{-1}(B_{\preceq})\subset \mathrm{Int}(B_{\preceq})$\,).

Suppose that $S=\mathbb{R}^2$ and $F$ is an orientation preserving
homeomorphism of $\mathbb{R}^2$ without fixed point, we say that
$\mathfrak{B}$ is \emph{free} if every brick $B\in\mathfrak{B}$ is
free. The stronger version of Franks' Lemma given in Guillou and Le
Roux \cite{Ler} asserts that there is no closed chain of bricks of
$\mathfrak{B}$ if $\mathfrak{B}$ is free. This implies that
$\preceq$ is a partial order. Furthermore, we can construct a
\emph{maximal free decomposition}: it is a brick decomposition with
free bricks such that the union of two adjacent bricks is no more
free \cite{A}. The decomposition being maximal, two adjacent bricks
are comparable. In fact, it appears that for every brick $B$, the
union of bricks $B'\succeq B$ adjacent to $B$ is non-empty, as is
the union of adjacent bricks $B'\preceq B$. This implies that
$B_{\succeq}$ and $B_{\preceq}$ are connected surfaces with
boundary. The fact that we are working with bricks implies that
$\partial(B_{\succeq})$ and $\partial(B_{\preceq})$ is a one
dimensional manifold; the inclusion $F(B_{\succeq})\subset
\mathrm{Int}(B_{\succeq})$ implies that every component of
$\partial(B_{\succeq})$ and $\partial(B_{\preceq})$ is a Brouwer
line. If $F$ commutes with $T$, we can further assume that the free
decomposition $\mathfrak{B}$ is invariant by $T$ (see \cite{P1} for
details). The relation $\preceq$ is $T$-equivalent: $B\preceq B'
\Leftrightarrow T(B)\preceq T(B')$. However a maximal $T$-invariant
free decomposition is not necessarily maximal among the free
decomposition. For example, there is no $T$-invariant free
decomposition for $T^k$ ($k\geq3$) which is maximal among the free
decompositions (see \cite{A} for details). Thus in this case, we can
not assert that two adjacent bricks are comparable and that
$B_{\succeq}$ and $B_{\preceq}$ are connected. But if there exists
$p\in \mathbb{Z}\setminus\{0\}$ such that $T^p(B)\in B_\succeq$
(resp. $T^p(B)\in B_\preceq$), then the set $\bigcup_{k\in
\mathbb{Z}}T^k(B_\succeq)$ (resp. $\bigcup_{k\in
\mathbb{Z}}T^k(B_\preceq)$) is closed connected surface with
boundary, which is a direct consequence of Proposition 2.7 in
\cite{P1}. Moreover, the image par $F$ (resp. $F^{-1}$) of
$\bigcup_{k\in \mathbb{Z}}T^k(B_\succeq)$ (resp. $\bigcup_{k\in
\mathbb{Z}}T^k(B_\preceq)$) is contained in its interior.
\smallskip

\section{Proof of the generalization of the line translation theorem}
In this section, we first state some crucial lemmas and then prove
the main result Theorem \ref{thm:the line translation theorem}. We
delay the proofs of these lemmas in the next section.
\smallskip
\begin{lem}\label{lem:commute and topological line}
Let $H_1,H_2,...,H_p$ be pairwise commuting essential homeomorphisms
of $\mathbb{R}^2$ and $(q_i)_{1\leq i\leq q}$ be a family of
positive integers. If there is an essential line $\Gamma$ in
$\mathbb{R}^2$ such that $\Gamma<H_i^{q_i}(\Gamma)$ for every
$i\in\{1,\cdots,\,p\}$, then there is an essential line $\Gamma'$ in
$\mathbb{R}^2$ such that $\Gamma'<H_i(\Gamma')$ for every
$i\in\{1,\cdots,p\}$.
\end{lem}

\begin{lem}\label{lem:Fq satisfies the intersection property}
If $f\in \mathrm{Homeo}_*^\wedge(\mathbb{A})$, then $f^q\in
\mathrm{Homeo}_*^\wedge(\mathbb{A})$ for any integer $q>1$.
\end{lem}

Lemmas \ref{lem:commute and topological line} and \ref{lem:Fq
satisfies the intersection property} are similar to Proposition 3.1
and Lemma 6.2 in \cite{B1}, respectively. Whilst, the following
lemma is a key and new result given in this paper, by which we can
directly prove Corollary \ref{thm:(0,1)} and thus we can furthermore
prove Theorem \ref{thm:the line translation theorem}.
\begin{lem}\label{lem:bounded condition}Let $f\in \mathrm{Homeo}_*^\wedge(\mathbb{A})$. We suppose that
$F$ is a lift of $f$ to $\mathbb{R}^2$ and that
$$\emptyset\neq \mathrm{Rot}(F)\subset [-q+2,q-2],$$
where $q\geq2$, then there exists an oriented essential line
$\gamma$ in $\mathbb{A}$ joining $S$ to $N$ that is lifted to an
essential line $\Gamma$ in $\mathbb{R}^2$ which satisfies
$$T^{-q}(\Gamma) < F(\Gamma) < T^{q}(\Gamma).$$
\end{lem}
\smallskip



\begin{proof}[Proof of the Theorem \ref{thm:the line translation theorem}]
Firstly, we prove the special case where $(p,q)=(0,1)$ and
$(p',q')=(1,1)$, that is Corollary \ref{thm:(0,1)}. By the
hypothesis, we can suppose that
$\mathrm{Rot}(F)\subset[\frac{1}{n},1-\frac{1}{n}]$ for some $n\in
\mathbb{N}$ large enough. Hence the closure of
$\mathrm{Rot}(F^{2n}\circ T^{-n})$ is contained in $[-n+2,n-2]$. The
map $f$ satisfies the intersection property, thus does $f^{2n}$ by
Lemma \ref{lem:Fq satisfies the intersection property}. By Lemma
\ref{lem:bounded condition}, there exists an oriented essential line
$\gamma$ in $\mathbb{A}$, joining $S$ to $N$, that is lifted to an
essential line $\Gamma$ in $\mathbb{R}^2$ which satisfies
$T^{-n}(\Gamma) < (F^{2n}\circ T^{-n})(\Gamma) < T^{n}(\Gamma)$.
This implies that $\Gamma < F^{2n}(\Gamma) < T^{2n}(\Gamma)$.

By Lemma \ref{lem:commute and topological line}, let $H_1=F$,
$H_2=F^{-1}\circ T$ and $q_1=q_2=2n$, we can construct an essential
line $\Gamma'$ in $\mathbb{R}^2$ which satisfies
$\Gamma'<F(\Gamma')$ and $\Gamma'<(F^{-1}\circ T)(\Gamma')$, so
$\Gamma'<F(\Gamma')<T(\Gamma')$. We deduce that
$\gamma'=\pi(\Gamma')$ is an essential line in $\mathbb{A}$
satisfying the conclusion of Corollary \ref{thm:(0,1)}.
\smallskip

Secondly, we turn to prove the general case. Let $\Phi=T^{-p}\circ
F^q$ and $\Psi=T^{p'}\circ F^{-q'}$. In Proposition 4.1 of
\cite{B1}, B\'{e}guin, Crovisier, Le Roux and Patou proved the
following result:

``Let $\gamma$ be an essential line in the annulus $\mathbb{A}$.
Assume that some lift $\Gamma$ of $\gamma$ is disjoint from its
images under the maps $\Phi$ and $\Psi$. Then the $q+q'-1$ first
iterates of $\gamma$ under $F$ are pairwise disjoint, and ordered as
the $q + q'-1$ first iterates of a vertical line under a rigid
rotation of angle $\rho\in]\frac{p}{q},\frac{p'}{q'}[$.''

Remark here that the proof of \cite{B1} is written in the context of
the closed annulus, but it also works in the open annulus setting.

Therefore, to prove the theorem, it is enough to find an essential
line $\gamma$ in $\mathbb{A}$, and a lift of $\gamma$ which is
disjoint from its images under $\Phi$ and $\Psi$.

By the properties of rotation number, we have that
$\mathrm{Cl}(\mathrm{Rot}(F^{qq'}\circ T^{-pq'}))\subset]0,1[$. By
Lemma \ref{lem:Fq satisfies the intersection property}, $f^{qq'}$
satisfies the intersection property since $f$ satisfies the
property. By Corollary \ref{thm:(0,1)} that we have proved above,
there exists an oriented essential line $\gamma'$ in $\mathbb{A}$
that is lifted to an essential line $\Gamma'$ in $\mathbb{R}^2$
which satisfies
\begin{equation}\label{eq:qq'}
    \Gamma' <(F^{qq'}\circ T^{-pq'})(\Gamma') < T(\Gamma').
\end{equation}
In particular, we have that $\Gamma'<(F^q\circ
T^{-p})^{q'}(\Gamma')=\Phi^{q'}(\Gamma')$ and $\Gamma'<T(\Gamma')$.
Acting $T^{pq'-qp'}$ on the formula (\ref{eq:qq'}) and observing
that $qp'-pq'=1$, we get
\begin{equation*}\label{eq:q'q}
   T^{-1}(\Gamma') <(F^{qq'}\circ T^{-qp'})(\Gamma') < \Gamma'.
\end{equation*}
In particular, we have that $\Gamma'<(F^{-q'}\circ
T^{p'})^{q}(\Gamma')=\Psi^q(\Gamma')$.

By Lemma \ref{lem:commute and topological line}, let $H_1=\Phi$,
$H_2=\Psi$, $H_3=T$, $q_1=q'$, $q_2=q$ and $q_3=1$, we can construct
an essential line $\Gamma''$ in $\mathbb{R}^2$ which satisfies
$\Gamma''<\Phi(\Gamma'')$, $\Gamma''<\Psi(\Gamma'')$ and
$\Gamma''<T(\Gamma'')$. Therefore, the essential line
$\gamma''=\pi(\Gamma'')$ satisfies the conclusion of the theorem. We
have completed the proof.
\end{proof}\smallskip

\section{proofs of the lemmas}
\begin{proof}[Proof of Lemma \ref{lem:commute and topological line}]
Let us say that the essential line $\Gamma$ is of type
$(q_1,q_2,\dots,q_p)$  if $\Gamma<H_i^{q_i}(\Gamma)$ for every
$i\in\{1,\dots,p\}$. We want to prove that the existence of an
essential line $\Gamma$ of type $(q_1,q_2,\dots,q_p)$ implies the
existence of an essential line $\Gamma'$ of type $(1,\dots,1)$. By a
simple induction argument, it is sufficient to prove the existence
of an essential line $\Gamma'$ of type $(1,q_2,\dots,q_p)$.

We choose some essential lines of $(\Gamma_i)_{0\leq i\leq q_1-1}$
in $\mathbb{R}^2$ such that
$$\Gamma=\Gamma_0<\Gamma_1<\cdots<\Gamma_{q_1-1}<H_j^{q_j}(\Gamma_0),
\quad \mathrm{for\,\, every}\,\, j\in\{1,\dots,p\}.$$

Consider the essential line
$$\Gamma'=H_1^{q_1}(\Gamma_0)\vee H_1^{q_1-1}(\Gamma_1)\vee\cdots\vee H_1(\Gamma_{q_1-1})=\bigvee_
{i=0}^{q_1-1}H_1^{q_1-i}(\Gamma_i).$$

For every $i\in\{0,\ldots,q_1-2\}$, we have $\Gamma'\leq
H_1^{q_1-i}(\Gamma_{i})$ (by the definition of $\Gamma'$ and by item
(2) of remark \ref{rem:v}) and
$H_1^{q_1-i}(\Gamma_{i})<H_1^{q_1-i}(\Gamma_{i+1})$. Hence for every
$i\in\{0,\ldots,q_1-2\}$, we get
$\Gamma'<H_1^{q_1-i}(\Gamma_{i+1})$. Moreover, we have $\Gamma'\leq
H_1(\Gamma_{q_1-1})$ and $\Gamma_{q_1-1}<H_1^{q_1}(\Gamma_{0})$.
Hence, $\Gamma'<H_1^{q_1+1}(\Gamma_{0})$. Finally, using item (1) of
remark \ref{rem:v}, we get
$$\Gamma'<\bigvee_{i=0}^{q_1-1}H_1^{q_1-i+1}(\Gamma_i)=H_1(\Gamma').$$

Observe that $\Gamma_i<H_j^{q_j}(\Gamma_i)$ for every
$i\in\{0,1,\cdots,q_1-1\}$ and $j\in\{2,\dots,p\}$. Therefore, we
have $\Gamma'\leq
H_1^{q_1-i}(\Gamma_i)<H_1^{q_1-i}(H_j^{q_j}(\Gamma_i))$. Using the
definition of $\Gamma'$, the fact that $H_j$ commutes with $H_1$,
and remark \ref{rem:v}, we have
$$\Gamma'<\bigvee_{i=0}^{q-1}H_1^{q_1-i}(H_j^{q_j}(\Gamma_i))=H_j^{q_j}(\Gamma').$$
Hence, the essential line $\Gamma'$ is of type $(1,q_2,\dots,q_p)$.
We have completed the proof.
\end{proof}

\begin{proof}[Proof of Lemma \ref{lem:Fq satisfies the intersection property}]
It is similar to the proof of Lemma \ref{lem:commute and topological
line}. For every essential circle $\gamma$ in $\mathbb{A}$, we
denote by $B(\gamma)$ the connected component of
$\mathbb{A}\setminus\gamma$ which is ``below $\gamma$'' (that is,
``containing'' the lower end $S$). Given two essential circles
$\gamma$ and $\gamma'$ in $\mathbb{A}$, the boundary of the
connected component of $B(\gamma)\cap B(\gamma')$ ``containing'' $S$
is an essential circle, which we denote by $\gamma\vee\gamma'$. We
write $\gamma<\gamma'$ if $\mathrm{Cl}(B(\gamma))\subset
B(\gamma')$.

To prove by contraction, we suppose that there exists a positive
integer $q$ and an essential circle $\gamma$ such that
$\gamma<f^q(\gamma)$ (by replacing $f$ by $f^{-1}$ if necessary). We
consider some essential circles $(\gamma_i)_{0\leq i\leq q-1}$ such
that
$$\gamma=\gamma_0<\gamma_1<\cdots<\gamma_{q-2}<\gamma_{q-1}<f^{q}(\gamma_0)$$
and we define
$$\gamma'=\bigvee_{i=0}^{q-1}f^{q-i}(\gamma_i).$$

Like in the proof of Lemma \ref{lem:commute and topological line},
we get $\gamma'<f(\gamma')$. In particular, we obtain an essential
circle $\gamma'$ that is disjoint from its image under $f$.
\end{proof}

To prove Lemma \ref{lem:bounded condition}, we need the following
lemma:

\begin{lem}\label{lem:Franks theorem 2.1}
Let $f\in \mathrm{Homeo}_*^\wedge(\mathbb{A})$ and $F$ be a lift of
$f$ to $\mathbb{R}^2$. Suppose that there exist two path connected
sets $X_1$ and $X_2$ in $\mathbb{R}^2$ satisfying
\begin{enumerate}
  \item $F(X_i)\cap X_i=\emptyset$ for $i=1,2$;
  \item $T^k(X_i)\cap X_i=\emptyset$ ($i=1,2$) for every
$k\in \mathbb{Z}\setminus\{0\}$;
  \item Either $X_1=X_2$ or $T^k(X_1)\cap X_2=\emptyset$ for every
$k\in \mathbb{Z}$;
  \item There exist positive integers $p_i,q_i$ ($i=1,2$)
such that $$F^{q_1}(X_1)\cap T^{p_1}(X_1)\neq\emptyset\quad
\mathrm{and}\quad F^{q_2}(X_2)\cap T^{-p_2}(X_2)\neq \emptyset.$$
\end{enumerate}
Then $F$ has a fixed point.
\end{lem}

\begin{proof}

Since $F^{q_1}(X_1)\cap T^{p_1}(X_1)\neq\emptyset$ and $
F^{q_2}(X_2)\cap T^{-p_2}(X_2)\neq \emptyset$, there exist points
$x_i,x_i'\in X_i$ ($i=1,2$) such that $F^{q_1}(x_1)= T^{p_1}(x_1')$
and $F^{q_2}(x_2)=T^{-p_2}(x_2')$. In the case where $X_1\neq X_2$,
we choose a segment $\Gamma_1$ (resp. $\Gamma_2$) of $X_1$ (resp.
$X_2$) that contains $x_1$ and $x_1'$ (resp. $x_2$ and $x_2'$). In
the case where $X_1=X_2$, we choose a finite tree included in $X_1$
which contains $x_i,x_i'$ ($i=1,2$) and we write $\Gamma_1=\Gamma_2$
for this tree. Then we can replace the couple $(X_1,X_2)$ by the
couple $(\Gamma_1,\Gamma_2)$ which satisfies the same property. The
items (2) and (3) of Lemma \ref{lem:Franks theorem 2.1} imply that
$\pi(\Gamma_1)$ and $\pi(\Gamma_2)$ are two disjoint segments or a
finite tree in $\mathbb{A}$.

To prove by contradiction, we suppose that $F$ has no fixed point.
We can find a maximal free brick decomposition $\mathfrak{B}$ of the
plane that is $T$-invariant such that $B_1$ and $B_2$ are bricks of
this decomposition and contain $\Gamma_1$ and $\Gamma_2$
respectively (note here that $B_1=B_2$ if $X_1=X_2$). We have that
$B_1\prec T^{p_1}(B_1) $ and $B_2\prec T^{-p_2}(B_2)$ by the
hypotheses where $B\prec B'$ means $B\preceq B'$ but $B'\npreceq B$.
This implies that, for every $n, n'\geq1$, we have
\begin{equation}\label{eq:B+B-}
    B_1\prec T^{np_1}(B_1)\quad \mathrm{and}\quad B_2\prec
T^{-n'p_2}(B_2).
\end{equation}

Now consider the sets
$$\bigcup_{k\in\mathbb{Z}}T^k({B_1}_{\succeq})=\bigcup_{k\in\mathbb{Z}}T^k({B_1})_{\succeq}
\quad\mathrm{and}\quad
\bigcup_{k\in\mathbb{Z}}T^k({B_2}_{\succeq})=\bigcup_{k\in\mathbb{Z}}T^k({B_2})_{\succeq}.$$
We claim that the inclusions
$B_2\subset\bigcup_{k\in\mathbb{Z}}T^k({B_1}_{\succeq})$ and
$B_1\subset\bigcup_{k\in\mathbb{Z}}T^k({B_2}_{\succeq})$ can not
happen simultaneously. Otherwise, there exist integers $k$ and $k'$,
satisfying that
\begin{equation}\label{eq:B+-}
   T^k(B_1) \preceq B_2\quad \mathrm{and}\quad T^{k'}(B_2)\preceq
B_1.
\end{equation}
From (\ref{eq:B+B-}) and (\ref{eq:B+-}), we get $
T^{k+n'p_2+k'}(B_1)\prec B_1$ for every $n'\geq1$. When $n'$ is
large enough we have $k''=k+k'+n'p_2>0$. Hence, we have
$T^{k''p_1}(B_1)\prec B_1$. By (\ref{eq:B+B-}), we also have
$B_1\prec T^{k''p_1}(B_1)$. It implies that $B_1\prec B_1$, which is
impossible. The fact that the inclusions
$B_2\subset\bigcup_{k\in\mathbb{Z}}T^k({B_1}_{\succeq})$ and
$B_1\subset\bigcup_{k\in\mathbb{Z}}T^k({B_2}_{\succeq})$ can not
happen simultaneously implies that
\begin{equation*}\label{attractor and repellor do not meet}
    \bigcup_{k\in\mathbb{Z}}T^k({B_1}_{\succeq})\cap
\bigcup_{k\in\mathbb{Z}}T^k({B_2}_{\preceq})=\emptyset
 \quad \mathrm{or}\quad \bigcup_{k\in\mathbb{Z}}T^k({B_1}_{\preceq})\cap
  \bigcup_{k\in\mathbb{Z}}T^k({B_2}_{\succeq})=\emptyset.
\end{equation*}

As $T^{p_1}(B_1)\in {B_1}_\succeq$ and $T^{p_2}(B_2)\in
{B_2}_\preceq$, we know that
$\bigcup\limits_{k\in\mathbb{Z}}T^{k}({B_1}_{\succeq})$ and
$\bigcup\limits_{k\in\mathbb{Z}}T^{k}({B_2}_{\preceq})$ are
connected. This implies they project by $\pi$ onto essential
connected surfaces (i.e. containing an essential circle)
${b_1}_\succeq$ and ${b_2}_\preceq$ with boundaries. So there exists
at least a connected component of the boundary of ${b_1}_\succeq$
that is an essential circle and is free for $f$, which is contrary
to the fact $f$ has the intersection property. We can also get a
contradiction in the case where
$\bigcup_{k\in\mathbb{Z}}T^k({B_1}_{\preceq})\cap
  \bigcup_{k\in\mathbb{Z}}T^k({B_2}_{\succeq})=\emptyset$.

Hence, $F$ has at least one fixed point, we complete the proof.
\end{proof}

\begin{proof}[Proof of Lemma \ref{lem:bounded condition}]
We know that the map $f$ satisfies the intersection property, that
the lift $F\circ T^{q-1}$ is fixed point free, and that there is a
positive recurrent point of $f$ with positive rotation number for
this new lift. By Theorem \ref{thm:GS}, there is an oriented
essential line $\gamma$ of $\mathbb{A}$ joining $S$ to $N$ that is
lifted to an essential Brouwer line $\Gamma$ of $F\circ T^{q-1}$
which satisfies
$$\Gamma<(F\circ T^{q-1})(\Gamma).$$
We can write $\Gamma<T(\Gamma)<(F\circ T^{q})(\Gamma)$. Write
$F'=F\circ T^{q}$, then we have
$$F'^{-n}(\Gamma)<T^{-n}(\Gamma)<\Gamma<T^n(\Gamma)<F'^{n}(\Gamma)\quad
\mathrm{for\,\,every}\quad n\geq1,$$ which implies that $F'$ is
conjugate to a translation. We consider the annulus $\mathbb{A}' =
\mathbb{R}^2/F'$ and the homeomorphism $t$ induced by $T$ on
$\mathbb{A}'$. The map $F^{-1}\circ T^{q}= F'^{-1}\circ T^{2q}$ is a
lift of $t^{\,2q}$ that is fixed point free. By Theorem
\ref{thm:GS}, there are three cases to consider:

\begin{enumerate}
  \item There exists an oriented essential line $\gamma'$ of
$\mathbb{A}'$ that is lifted to an oriented line $\Gamma'$ in
$\mathbb{R}^2$ such that $\Gamma'<F'(\Gamma')$ and
$\Gamma'<(F^{-1}\circ T^{q})(\Gamma')$;
  \item There exists an oriented essential line $\gamma'$ of
$\mathbb{A}'$ that is lifted to an oriented line $\Gamma'$ in
$\mathbb{R}^2$ such that $\Gamma'<F'(\Gamma')$ and $(F^{-1}\circ
T^{q})(\Gamma')<\Gamma'$;
  \item There exists an essential circle in $\mathbb{A}'$ that is free
for $t^{\,2q}$.
\end{enumerate}
We will get the lemma in the first case and contradictions in the
two other cases.
\bigskip

In the case (1), we have
\begin{equation}\label{ineq:Gamma'' and TqGamma''}
    T^{-q}(\Gamma')<F(\Gamma')<T^{q}(\Gamma').
\end{equation}
Applying $T^{q}$ to (\ref{ineq:Gamma'' and TqGamma''}), we get
$\Gamma'<F'(\Gamma')<T^{2q}(\Gamma')$. Thus we have
\begin{equation}\label{eq:Gamma' is essential}
    T^{-2qn}(\Gamma')<F'^{-n}(\Gamma')<\Gamma'<F'^n(\Gamma')<T^{2qn}(\Gamma')
\quad \mathrm{for\,\,every}\quad n\geq1.
\end{equation}

Observe that $\bigcup\limits_{n\geq1}R(F'^{-n}(\Gamma'))\cap
L(F'^n(\Gamma'))=\mathbb{R}^2$ since $F'$ is conjugate to a
translation and the line $\gamma'=\pi(\Gamma')$ is an essential line
in $\mathbb{A}'$. For any $P>0$, by the compactness of
$[-q,q]\times[-P,P]$, there exists a positive integer $N$ such that
$[-q,q]\times[-P,P]\subset R(T^{-2Nq}(\Gamma'))\cap
L(T^{2Nq}(\Gamma'))$. We deduces that
$[-q,+\infty[\times[-P,P]\subset R(T^{-2Nq}(\Gamma'))$ and
$]-\infty,q]\times[-P,P]\subset L(T^{2Nq}(\Gamma'))$, and hence
$[-q+2Nq,+\infty[\times[-P,P]\subset R(\Gamma')$ and
$]-\infty,q-2Nq[\times[-P,P]\subset L(\Gamma')$. This implies that
$\Gamma'$ is essential.


By Lemma \ref{lem:commute and topological line}, let $H_1=T$,
$H_2=T^{q}\circ F^{-1}$, $H_3=T^{q}\circ F$, $q_1=2q$ and
$q_2=q_3=1$, we can construct an essential line $\Gamma''$ which
satisfies

$$T^{-q}(\Gamma'')<F(\Gamma'')<T^{q}(\Gamma'')$$ and
$$\Gamma''< T(\Gamma'').$$
Hence, $\gamma''=\pi(\Gamma'')$ is an essential line of $\mathbb{A}$
that satisfies the conclusion of the lemma.
\bigskip

In the case (2), we will see how to get a contradiction.

The line $\Gamma'$ satisfies $\Gamma'<(F\circ T^q)(\Gamma')$ and
$\Gamma'<(F\circ T^{-q})(\Gamma')$. We define the set $X$ of couple
of integers $(m,n)$ such that $\Gamma'<(F^m\circ T^n)(\Gamma')$. It
is a set stable by addition that contains $(1,q)$ and $(1,-q)$. So
it contains all the integers $(m+n, q(m-n))$, where $m\geq 0$,
$n\geq 0$ and at least one is non zero. This is exactly the set of
couples $(m, nq)$ where $m>0$ and $\vert n\vert \leq m$. For every
couple $(m,n)$ of integers such that $m>0$ and $\vert
\frac{n}{m}\vert\leq q-\frac{2}{m}$, we have $qm-1>0$ and $\vert
n-1\vert\leq\vert n\vert+1\leq qm-1$. Therefore, we have
$q(m,n)-(1,q)\in X$, which means
\begin{equation}\label{ineq:Gamma'}
    \Gamma'<(F\circ T^q)(\Gamma')< (F^m\circ T^n)^{q}(\Gamma').
\end{equation}

Fix $z\in\mathrm{Rec}^+(f)$ having a rotation number $\rho$ and
$\widetilde{z}\in \pi^{-1}(z)$. Since $F'=F\circ T^q$ is conjugate
to a translation and the line $\gamma'$ is an essential line in
$\mathbb{A}'$, we can always suppose that $\widetilde{z}$ is
contained in the region $\mathrm{Cl}(R(\Gamma'))\cap L((F\circ
T^q)(\Gamma'))$ by replacing $\Gamma'$ with an iterate
$F'^k(\Gamma')$ if necessary. Consider the homeomorphism $f_q$ of
the annulus $\mathbb{A}_q=\mathbb{R}^2/T^q$ lifted by $F$ and write
$\pi_q:\mathbb{R}^2\rightarrow \mathbb{A}_q$ for the covering map.
We know that $\mathrm{Rec}^+(f)=\mathrm{Rec}^+(f^q)$ and we can
prove similarly that $\pi_q(\widetilde{z})\in \mathrm{Rec}^+(f_q^q)$
(we give a proof in the Appendix, see Lemma \ref{lem:rec(fqq)}). In
other words, there exist two sequences of integers $(n_i)_{i\geq1}$
and $(m_i)_{i\geq1}$ such that $m_i\rightarrow+\infty$ and
 $F^{qm_i}\circ T^{qn_i}(\widetilde{z})\rightarrow\widetilde{z}$
as $i\rightarrow+\infty$. Certainly, $\frac{qn_i}{qm_i}\rightarrow
-\rho$ as $i\rightarrow+\infty$. Therefore, there is a positive
integer $N$ such that when $i\geq N$, we have
$\vert\frac{n_i}{m_i}\vert<q-1<q-\frac{2}{m_i}$.

By the inequation (\ref{ineq:Gamma'}), we have $\Gamma'<(F\circ
T^q)(\Gamma')< (F^{m_i}\circ T^{n_i})^{q}(\Gamma')$ when $i\geq N$.
On one hand, the points of the sequence $\{(F^{m_i}\circ
T^{n_i})^{q}(\widetilde{z})\}_{i\geq N}$ belong to $R(F\circ
T^q(\Gamma'))$ and so their limit belongs to $\mathrm{Cl}(R(F\circ
T^q(\Gamma')))$. On the other hand, the limit of the sequence
belongs to $L(F\circ T^q(\Gamma'))$, which is a contradiction.
\bigskip

In the case (3), we will get a contradiction again.

By Lemma \ref{lem:Fq satisfies the intersection property}, we deduce
that there exists an essential circle $\gamma'$ free for $t$.
Therefore $\gamma'$ lifts to a Brouwer line $\Gamma'$ for $T$, in
particular, $\Gamma'\cap T(\Gamma')=\emptyset$.

The curve $\gamma'$ is closed in $\mathbb{A}'$, which implies that
its lift $\Gamma'$ satisfies
$$\Gamma'=F'(\Gamma')=(F\circ T^{q})(\Gamma').$$
Hence
$$T^{-q}(\Gamma')=F(\Gamma')\quad \mathrm{and}\quad F(\Gamma')\cap
T^{1-q}(\Gamma')=\emptyset.$$

Considering the map $F''=F\circ T^{q-1}$, we have the following:
\begin{itemize}
\item there exists a free line $\Gamma'$ of $\mathbb{R}^2$ for $F''$ and for $T$ (and hence for $T^k$ for every $k\neq0$) such that
$T^{-1}(\Gamma')=F''(\Gamma')$;
\item $\emptyset\neq\mathrm{Rot}(F'')\subset
[1,2q-3]$, in particular, $F''$ has no fixed point.
\end{itemize}
The first item above implies that there exists a segment
$\Gamma_0\subset \Gamma'$ such that $F''(\Gamma_0)\cap
T^{-1}(\Gamma_0)\not=\emptyset$ and that $\Gamma_0$ is free both for
$F''$ and for $T^k$ for any $k\in \mathbb{Z}\setminus\{0\}$.

The second item above implies that there exists $z\in\mathbb{A}$
such that $z\in \mathrm{Rec}^+(f)$ and $\rho(F'';z)\in [1,2q-3]$. We
claim that $z\notin \pi(\Gamma_0)$. Otherwise, choose a lift
$\widetilde{z}$ of $z$ in $\Gamma_0$. Let $U$ be an open disk
containing $\widetilde{z}$, located in the region between
$T^{-1}(\Gamma')$ and $T(\Gamma')$. There exist positive integers
$n\geq2$, $l\geq1$ such that $F''^n(\widetilde{z})\in T^l(U)$. As we
have $F''^n(\widetilde{z})\in F''^n(\Gamma_0)\subset
T^{-n}(\Gamma')$, we deduce that $T^l(U)\cap
T^{-n}(\Gamma')\neq\emptyset$ which is impossible. We have completed
the claim. As $z\notin \pi(\Gamma_0)$, we can find a disk $U'$ free
for $F''$, containing
 $\widetilde{z}$, disjoint from $\Gamma_0$, and satisfying
  $T^k(U')\cap U'=\emptyset$ for every $k\neq 0$ and $F''^n(U')\cap T^l(U')\neq \emptyset$.

By Lemma \ref{lem:Franks theorem 2.1}, $F''$ has a fixed point,
which is impossible.
\end{proof}
\smallskip

In the rest of this section, we give some remarks about the
relations between the positively recurrent set and the rotation
number set.

Let $f\in \mathrm{Homeo}_*(\mathbb{A})$ and $F$ be a lift of $f$ to
$\mathbb{R}^2$. Suppose that $z\in \mathrm{Rec}^+(f)$ and
$\widetilde{z}\in \pi^{-1}(z)$. We define $\mathcal
{E}(z)\subset\mathbb{R}\cup\{-\infty,+\infty\}$ by saying that
$\rho\in\mathcal {E}(z)$ if there exists a sequence
$\{n_k\}_{k=1}^{+\infty}\subset\mathbb{N}$ such that
\begin{eqnarray*}
   &\bullet& \lim_{k\rightarrow+\infty}f^{n_k}(z)=z;\qquad\qquad\qquad\qquad\qquad\qquad\qquad\qquad
   \qquad\qquad\qquad\qquad\qquad\qquad\quad \\
   &\bullet& \lim_{k\rightarrow+\infty}\frac{p_1(F^{n_k}(\widetilde{z}))-p_1(\widetilde{z})}{n_k}=\rho.
\end{eqnarray*}

Define $\rho^-(F;z)=\inf\mathcal
{E}(z)\quad\mathrm{and}\quad\rho^+(F;z)=\sup\mathcal {E}(z).$
Obviously, we have that $\rho(F;z)$ exists if and only if
$\rho^-(F;z)=\rho^+(F;z)\in\mathbb{R}$.\bigskip

The following proposition is due to Franks \cite{F} when
$\mathbb{A}$ is closed annulus and $f$ has no wandering point, and
it was improved by Le Calvez \cite{P1} when $\mathbb{A}$ is open
annulus and $f$ satisfies the intersection property. We can use
Lemma \ref{lem:Franks theorem 2.1} to prove it.

\begin{prop}\label{thm:FP}Let $f\in \mathrm{Homeo}_*^\wedge(\mathbb{A})$
 and $F$ be a lift of $f$ to $\mathbb{R}^2$. Suppose that there
exist two recurrent points $z_{1}$ and $z_{2}$ such that
$-\infty\leq\rho^-(F;z_{1})<\rho^+(F;z_{2})\leq+\infty$. Then for
any rational number $p/q\in ]\rho^-(F;z_{1}),\rho^+(F;z_{2})[$
written in an irreducible way, there exists a periodic point of
period $q$ whose rotation number is $p/q$.
\end{prop}
\begin{proof}We consider the map $f^q$ and its lift $F'=F^q\circ
T^{-p}$. It is sufficient to prove that $F'$ has a fixed point.

By Lemma \ref{lem:Fq satisfies the intersection property}, $f^q$
satisfies the intersection property. By the properties of the
rotation number, we have $\rho^-(F';z_1)<0<\rho^+(F';z_2)$. For any
fixed point $z$ of $f^q$, we have $\rho^-(F';z)=\rho^+(F';z)\in
\mathbb{Z}$. Therefore, the points $z_1$ and $z_2$ are not fixed
points of $f^q$. So we can choose open disks $b_1$ and $b_2$,
containing respectively $z_1$ and $z_2$, equal if $z_1=z_2$,
disjoint if $z_1\neq z_2$, such that the connected components of
$\pi^{-1}(b_1)$ and $\pi^{-1}(b_2)$ are free for $F'$. Choose a
component $\widetilde{b}_1$ of $\pi^{-1}(b_1)$ and a component
$\widetilde{b}_2$ of $\pi^{-1}(b_2)$ (equal to $\widetilde{b}_1$ if
$z_1=z_2$). Observe that $\widetilde{b}_1$ and $\widetilde{b}_2$
satisfy all the hypotheses of Lemma \ref{lem:Franks theorem 2.1}.
Then, Proposition \ref{thm:FP} follows from Lemma \ref{lem:Franks
theorem 2.1}.
\end{proof}
From the Lemma \ref{lem:Franks theorem 2.1} and Proposition
\ref{thm:FP}, we have the following corollaries:
\begin{cor}\label{cor:rec not empty implies rot neq empty}
Let $f\in \mathrm{Homeo}_*^\wedge(\mathbb{A})$ and $F$ be a lift of
$f$ to $\mathbb{R}^2$. If $\mathrm{Rec}^+(f)\neq \emptyset$, then
$\mathrm{Rot}(F)\neq\emptyset$.
\end{cor}
\begin{proof}Let $z\in \mathrm{Rec}^+(f)$ and $\widetilde{z}$ be a lift of
$z$. If $z$ is a fixed point, the corollary is true. Otherwise,
choose a free disk $b$ for $f$ in $\mathbb{A}$ that contains $z$ and
a connected component $\widetilde{b}$ of $\pi^{-1}(b)$.

If $\rho^-(F;z)\neq\rho^+(F;z)$, then $\mathrm{Rot}(F)\neq\emptyset$
by Proposition \ref{thm:FP}. If $\rho^-(F;z)=\rho^+(F;z)\in
\mathbb{R}$, then $z$ has a rotation number and therefore
$\mathrm{Rot}(F)\neq\emptyset$. It remains to study the cases
$\rho^-(F;z)=\rho^+(F;z)=+\infty$ and
$\rho^-(F;z)=\rho^+(F;z)=-\infty$. They are similar. We will explain
now why the condition $\rho^-(F;z)=\rho^+(F;z)=+\infty$ implies that
$f$ has a fixed point and therefore that
$\mathrm{Rot}(F)\neq\emptyset$. We can find two rational numbers
$\frac{p_i}{q_i}$ $(i=1,2)$ and an integer $n$, arbitrarily large,
such that $F^{q_1}(\widetilde{b})\cap
T^{p_1}(\widetilde{b})\neq\emptyset$ and $F^{q_2}(\widetilde{b})\cap
T^{p_2}(\widetilde{b})\neq \emptyset$. Consider the lift $F'=F\circ
T^{-n}$ of $f$ to $\mathbb{R}^2$, we have
$F'^{q_1}(\widetilde{b})\cap
T^{p_1-q_1n}(\widetilde{b})\neq\emptyset$ and
$F'^{q_2}(\widetilde{b})\cap T^{p_2-q_2n}(\widetilde{b})\neq
\emptyset$. Recall that $\widetilde{b}$ is free for $F'$. By Lemma
\ref{lem:Franks theorem 2.1}, we have that $F'$ has a fixed point
and hence $f$ has a fixed point. We have completed the proof.
\end{proof}
\begin{cor}\label{cor:irrational}
Let $f\in \mathrm{Homeo}_*^\wedge(\mathbb{A})$ and $F$ a lift of $f$
to $\mathbb{R}^2$. The set $\mathrm{Rot}(F)$ is reduced to a single
irrational number $\rho$ if and only if
$\mathrm{Rec}^+(f)\neq\emptyset$ and $f$ has no periodic orbit.
\end{cor}\smallskip

\section{Weak rotation number and the line translation theorem}
We have supposed previously that there exists a recurrent point. It
is natural to wonder if the main theorem is still true without this
hypothesis.

In this section, we define the weak rotation set of $f$ which is a
generalization of the rotation set defined in Section 1. It was
introduced in \cite{F2} in a more generalized framework (see also
\cite{Lerh} for a local study). We first give some properties of
weak rotation numbers, then we prove the generalization of the line
translation theorem.\smallskip

Fix $f\in\mathrm{Homeo}_*^\wedge(\mathbb{A})$ and a lift $F$. For
every compact set $K$ define the set
$\mathrm{Rot}_{\mathrm{weak},K}(F)$ in the following way :
$r\in[-\infty,+\infty]$ belongs to
$\mathrm{Rot}_{\mathrm{weak},K}(F)$ if there exists a sequence
$\{\widetilde{z}_k\}_{k\geq1}\subset \mathbb{R}^2$ and a sequence
$\{n_k\}_{k\geq1}$ of positive integers such that
\begin{eqnarray*}
   &\bullet& \pi(\widetilde{z}_k)\in K;\qquad\qquad\qquad\qquad\qquad\qquad\qquad\qquad
   \qquad\qquad\qquad\qquad\qquad\qquad\quad \\
   &\bullet& \pi(F^{n_k}(\widetilde{z}_k))\in K; \\
   &\bullet& \lim_{k\rightarrow+\infty}n_k=+\infty; \\
   &\bullet& \lim_{k\rightarrow+\infty}\frac{p_1(F^{n_k}(\widetilde{z}_k))-p_1(\widetilde{z}_k)}{n_k}=r.
\end{eqnarray*}
Define the weak rotation set of $F$ as
following$$\mathrm{Rot}_{\mathrm{weak}}(F)=\bigcup_{K\in
\mathrm{C}(\mathbb{A})}\mathrm{Rot}_{\mathrm{weak},K}(F),$$ where
$\mathrm{C}(\mathbb{A})$ is the collection of all compact subsets of
$\mathbb{A}$.
\begin{prop}\label{prop:properties of weak rotation number}If $f\in\mathrm{Homeo}_*^\wedge(\mathbb{A})$, we have
the following properties:
\begin{enumerate}
  \item $\mathrm{Rot}_{\mathrm{weak},K}(F)$ is closed;
  \item If K separates the two ends $N$ and $S$ of $\mathbb{A}$, then $\mathrm{Rot}_{\mathrm{weak},K}(F)\neq\emptyset$.
  Therefore $\mathrm{Rot}_{\mathrm{weak}}(F)\neq\emptyset$;
  \item $\mathrm{Rot}(F)\subset\mathrm{Rot}_{\mathrm{weak}}(F)$;\item For every $p\in\mathbb{Z}$
   and $q\in\mathbb{Z}$, we have $\mathrm{Rot}_{\mathrm{weak}}(F^q\circ T^p)
  =q\mathrm{Rot}_{\mathrm{weak}}(F)+p$.
\end{enumerate}
\end{prop}
\begin{proof}
(1) For every $n\geq1$ define a set $R_n\subset \mathbb{R}$ as
$$R_n=\left\{\frac{p_1\circ
F^n(\widetilde{z})-p_1(\widetilde{z})}{n}\,\Big\vert\,\widetilde{z}\in\pi^{-1}(K)\quad
\mathrm{and\quad}F^n(\widetilde{z})\in\pi^{-1}(K)\right\}.$$ Observe
now that
$$\mathrm{Rot}_{\mathrm{weak},K}(F)=\bigcap_{n\geq1}\mathrm{Cl}(\bigcup_{k\geq n}R_k),$$
so $\mathrm{Rot}_{\mathrm{weak},K}(F)$ is closed.\bigskip

(2) To prove that $\mathrm{Rot}_{\mathrm{weak},K}(F)$ is not empty,
it is sufficient to prove that for every $q\geq1$, there is a point
$z_q\in K$ such that $f^q(z_q)\in K$. Suppose that there exists
$q\geq1$ such that $K\cap f^q(K)=\emptyset$. As $K$ is compact, we
can find an open neighborhood $U$ of $K$ such that $U\cap
f^q(U)=\emptyset$. Since $K$ separates the two ends of $\mathbb{A}$,
we deduce that $U$ contains an essential circle. This contradicts
Lemma \ref{lem:Fq satisfies the intersection property}. \bigskip

(3) For every $r\in \mathrm{Rot}(F)$, there exists $z\in
\mathrm{Rec}^+(f)$ such that $\rho(F;z)=r$, we choose $K$ to be a
closed disk whose interior contains $z$. By definitions, we have
$r=\rho(F;z)\in\mathrm{Rot}_{\mathrm{weak},K}(F)\subset
\mathrm{Rot}_{\mathrm{weak}}(F)$. Therefore,
$\mathrm{Rot}(F)\subset\mathrm{Rot}_{\mathrm{weak}}(F)$. \bigskip

(4) Since $F\circ T=T\circ F$, we clearly have
$\mathrm{Rot}_{\mathrm{weak}}(F\circ
T^p)=\mathrm{Rot}_{\mathrm{weak}}(F)+p$. Let us prove now that
$\mathrm{Rot}_{\mathrm{weak}}(F^q)=q\mathrm{Rot}_{\mathrm{weak}}(F)$
for every $q\in \mathbb{Z}$. When $q=0$, it is trivial. First, fix
$q>0$. Obviously, $\mathrm{Rot}_{\mathrm{weak},K}(F^q)\subset
q\mathrm{Rot}_{\mathrm{weak},K}(F)$ for any $K\in
\mathrm{C}(\mathbb{A})$ by definition. Therefore,
$\mathrm{Rot}_{\mathrm{weak}}(F^q)\subset
q\mathrm{Rot}_{\mathrm{weak}}(F)$.

Now, we prove the inverse. For every $r\in
\mathrm{Rot}_{\mathrm{weak}}(F)$, there exists $K\in
\mathrm{C}(\mathbb{A})$ such that $r\in
\mathrm{Rot}_{\mathrm{weak},K}(F)$. Therefore, there exists a
sequence $\{\widetilde{z}_k\}_{k\geq1}\subset \mathbb{R}^2$ and a
sequence $\{n_k\}_{k\geq1}$ of positive integers such that
\begin{eqnarray*}
   &\bullet& \pi(\widetilde{z}_k)\in K;\qquad\qquad\qquad\qquad\qquad\qquad\qquad\qquad
   \qquad\qquad\qquad\qquad\qquad\qquad\quad \\
   &\bullet& \pi(F^{n_k}(\widetilde{z}_k))\in K; \\
   &\bullet& \lim_{k\rightarrow+\infty}n_k=+\infty; \\
   &\bullet& \lim_{k\rightarrow+\infty}\frac{p_1(F^{n_k}(\widetilde{z}_k))-p_1(\widetilde{z}_k)}{n_k}=r.
\end{eqnarray*}
Write $n_k=l_kq+p_k$ where $0\leq p_k <q$. Suppose that there are
infinitely many $k$ such that $p_k=p$ where $0\leq p<q$. We can
choose a subsequence $\{n_{k_j}\}_{j\geq1}$ of $\{n_k\}_{k\geq1}$
such that $n_{k_j}=l_{k_j}q+p$ and
$$\lim_{j\rightarrow+\infty}\frac{p_1(F^{ql_{k_j}}(F^p(\widetilde{z}_{k_j})))-p_1(F^p(\widetilde{z}_{k_j}))}{ql_{k_j}}=r.$$
Then we have
$$qr=\lim\limits_{j\rightarrow+\infty}\frac{p_1(F^{ql_{k_j}}(F^p(\widetilde{z}_{k_j})))-p_1(F^p(\widetilde{z}_{k_j}))}{l_{k_j}}\in
\mathrm{Rot}_{\mathrm{weak},f^p(K)}(F^q).$$ Hence
$q\mathrm{Rot}_{\mathrm{weak}}(F)\subset\mathrm{Rot}_{\mathrm{weak}}(F^q)$.\smallskip

For the case $q<0$, it is sufficient to prove that
$\mathrm{Rot}_{\mathrm{weak},K}(F^{-1})=-\mathrm{Rot}_{\mathrm{weak},K}(F)$
for every $K\in\mathrm{C}(\mathbb{A})$. For every
$r\in\mathrm{Rot}_{\mathrm{weak},K}(F)$, by letting
$\widetilde{z}\,'_k=F^{n_k}(\widetilde{z}_k)$ in the definition of
$\mathrm{Rot}_{\mathrm{weak},K}(F)$, we have
\begin{eqnarray*}
   &\bullet& \pi(\widetilde{z}\,'_k)\in K;\qquad\qquad\qquad\qquad\qquad\qquad\qquad\qquad
   \qquad\qquad\qquad\qquad\qquad\qquad\qquad \\
   &\bullet& \pi(F^{-n_k}(\widetilde{z}\,'_k))\in K; \\
   &\bullet& \lim_{k\rightarrow+\infty}n_k=+\infty; \\
   &\bullet& \lim_{k\rightarrow+\infty}\frac{p_1(F^{-n_k}(\widetilde{z}\,'_k))-p_1(\widetilde{z}\,'_k)}{n_k}=-r.
\end{eqnarray*}
Therefore,
\begin{equation}\label{eq:minus weak rotation set}-\mathrm{Rot}_{\mathrm{weak},K}(F)\subset
\mathrm{Rot}_{\mathrm{weak},K}(F^{-1}).\end{equation} By replacing
$F$ by $F^{-1}$ in the conclusion (\ref{eq:minus weak rotation
set}), we have $\mathrm{Rot}_{\mathrm{weak},K}(F^{-1})\subset
-\mathrm{Rot}_{\mathrm{weak},K}(F)$. We have completed the proof.
\end{proof}

\begin{lem}\label{lem:gamma < f'gamma, then [0,+infty]}
Suppose that there exists a compact set $K\subset \mathbb{A}$ and an
orientated line $\Gamma$ of $\mathbb{R}^2$ satisfying
\begin{enumerate}
\item $\Gamma<T(\Gamma)$;
\item $\Gamma<F(\Gamma)$ (resp. $F(\Gamma)<\Gamma$);
\item $\pi^{-1}(K)\subset \bigcup_{k\in\mathbb{Z}}T^k(\mathrm{Cl}(R(\Gamma))\cap
L(T(\Gamma)))$.
\end{enumerate}
Then $\mathrm{Rot}_{\mathrm{weak},K}(F)\subset[0,+\infty]$ (resp.
$\mathrm{Rot}_{\mathrm{weak},K}(F)\subset[-\infty,0]$). As a
consequence, if $\Gamma$ is a lift of an oriented essential line in
$\mathbb{A}$ joining $S$ to $N$, satisfying $\Gamma<F(\Gamma)$
(resp. $F(\Gamma)<\Gamma$), then
$\mathrm{Rot}_{\mathrm{weak}}(F)\subset[0,+\infty]$ (resp.
$\mathrm{Rot}_{\mathrm{weak}}(F)\subset[-\infty,0]$).
\end{lem}
\begin{proof}
We will make a proof by contradiction, we suppose that
$\Gamma<T(\Gamma)$ and that there exists a negative number $\rho$
such that $\rho\in \mathrm{Rot}_{\mathrm{weak},K}(F)$. Write
$$\widetilde{K}=\pi^{-1}(K)\cap \mathrm{Cl}(R(\Gamma))\cap
L(T(\Gamma))$$ and
$$\mathrm{width}(\widetilde{K})=\sup_{\widetilde{z}_1,\widetilde{z}_2\in
\widetilde{K}}\{p_1(\widetilde{z}_1)-p_1(\widetilde{z}_2)\}.$$

We first claim that $\mathrm{width}(\widetilde{K})$ is finite.
According to the hypothesis (3), we have $K=\pi(\widetilde{K})$. For
every $k\geq1$, we define an open set
$$\widetilde{U}_k=\left(R(T^{-1}(\Gamma))\cap L(T(\Gamma))\right)\cap\{(x,y)\in\mathbb{R}^2\mid -k<x<k\}.$$
By the hypothesis (1), the sequence of open sets
$\{\pi(\widetilde{U}_k)\}_{k\geq1}$ in $\mathbb{A}$ is increasing.
Since $\mathrm{Cl}(R(\Gamma))\cap L(T(\Gamma))\subset
R(T^{-1}(\Gamma))\cap L(T(\Gamma))$, we have
$\widetilde{K}\cap\{(x,y)\in\mathbb{R}^2\mid -k<x<k\}\subset
\widetilde{U}_k$. Obviously, $\widetilde{K}\subset
\bigcup_{k\geq1}\widetilde{U}_k$ and $K\subset
\bigcup_{k\geq1}\pi(\widetilde{U}_k)$. As $K$ is compact, there is a
positive integer $N$ such that $K\subset\pi(\widetilde{U}_N)$. It
implies that, for every $\widetilde{z}\in\widetilde{K}$, there
exists $k\in\mathbb{Z}$ such that
$T^k(\widetilde{z})\in\widetilde{U}_N$. Observe that
$$T^l(R(T^{-1}(\Gamma))\cap L(T(\Gamma)))\cap(\mathrm{Cl}(R(\Gamma))\cap L(T(\Gamma)))
=\emptyset\quad \mathrm{if}\quad l\in\mathbb{Z}\setminus\{0,1\}.$$
It implies that the only possibilities for $k$ are $-1$ and $0$. So
$\widetilde{K}\subset \widetilde{U}_N\cup T(\widetilde{U}_N)$. It
completes the claim.
\smallskip

By the definition of $\mathrm{Rot}_{\mathrm{weak},K}(F)$, there is a
sequence $\{\widetilde{z}_{k}\}_{k\geq1}\subset \widetilde{K}$ and a
sequence $\{n_{k}\}_{k\geq1}$ of positive integers such that
\begin{eqnarray*}
   &\bullet& \lim_{k\rightarrow+\infty}n_k=+\infty; \\
   &\bullet& \{\pi(F^{n_{k}}(\widetilde{z}_{k}))\}_{k\geq1}\subset K; \qquad\qquad\qquad\qquad\qquad\qquad\qquad
   \qquad\qquad\qquad\qquad\qquad\\
   &\bullet&
\lim_{k\rightarrow+\infty}\frac{p_1(F^{n_k}(\widetilde{z}_k))-p_1(\widetilde{z}_k)}{n_k}=\rho.
\end{eqnarray*}

If $k$ is large enough, then
$$p_1(F^{n_{k}}(\widetilde{z}_{k}))-p_1(\widetilde{z}_{k})<-2\mathrm{width}(\widetilde{K}).$$ It implies that  $F^{n_{k}}(\widetilde{z}_{k})$ is
on the left of $\Gamma$. But, by the hypothesis (2), for every
$n\geq1$ and $\widetilde{z}\in \widetilde{K}$, $F^n(\widetilde{z})$
is on the right of $\Gamma$, we get a contradiction. \smallskip

Let us prove now the second statement of the lemma.  The line
$\Gamma$ is a lift of an oriented essential line in $\mathbb{A}$
joining $S$ to $N$. Observe that
$\bigcup_{k\in\mathbb{Z}}T^k(\mathrm{Cl}(R(\Gamma))\cap
L(T(\Gamma)))=\mathbb{R}^2$, then we get the last consequence.
\end{proof}
Using Lemma \ref{lem:gamma < f'gamma, then [0,+infty]}, we can get
more properties about the weak rotation number set of $F$.

\begin{prop}\label{prop:more properties of the weak rotation number}We have the following properties:
\begin{enumerate}
 \item If $\frac{p}{q}$ is given and if
$\mathrm{Rot}_{\mathrm{weak},K}(F)$ contains $r_1$ and $r_2$ where
$r_1 < \frac{p}{q}<r_2$, then there exists $\widetilde{z}$ such that
$F^q(\widetilde{z})=T^p(\widetilde{z})$.
  \item The sets
$\mathrm{Cl}(\mathrm{Rot}_{\mathrm{weak}}(F))$ and
$\mathrm{Cl}(\mathrm{Rot}(F))$ are the same intervals except in the
case where $\mathrm{Rec}^+(f)=\emptyset$ and
$\mathrm{Rot}_{\mathrm{weak}}(F)$ is reduced to an element of
$\mathbb{R}$.
 \end{enumerate}

\end{prop}

\begin{proof}
To prove (1), consider the map $f^q$ and its lift $F'=F^q\circ
T^{-p}$. We must prove that $F'$ has a fixed point. We suppose that
$F'$ has no fixed point. By Lemma \ref{lem:Fq satisfies the
intersection property}, $f^q$ satisfies the intersection property.
By Theorem \ref{thm:GS}, there exists an essential line of
$\mathbb{A}$ joining $S$ to $N$ that lifts to a Brouwer line
$\Gamma$ of $F'$. Then either $\Gamma<F'(\Gamma)$ or
$F'(\Gamma)<\Gamma$. By Lemma \ref{lem:gamma < f'gamma, then
[0,+infty]}, we have either
$\mathrm{Rot}_{\mathrm{weak}}(F)\subset[0,+\infty]$ or
$\mathrm{Rot}_{\mathrm{weak}}(F)\subset[-\infty,0]$. However, by the
assertion (4) of Proposition \ref{prop:properties of weak rotation
number}, we have $qr_1-p\in\mathrm{Rot}_{\mathrm{weak}\,K}(F')$ and
$qr_2-p\in\mathrm{Rot}_{\mathrm{weak}\,K}(F')$ with
$qr_1-p<0<qr_2-p$, which is impossible. Therefore, $F'$ has a fixed
point.
\smallskip

By the assertion (3) of Proposition \ref{prop:properties of weak
rotation number}, we know that the inclusion
$\mathrm{Cl}(\mathrm{Rot}(F))\subset\mathrm{Cl}(\mathrm{Rot_{\mathrm{weak}}}(F))$
is true. If $\mathrm{Rot}_{\mathrm{weak}}(F)$ is reduced to a real
number and if $\mathrm{Rec}^+(f)\neq\emptyset$, then by Corollary
\ref{cor:rec not empty implies rot neq empty}, $\mathrm{Rot}(F)$ is
not empty and therefore is equal to
$\mathrm{Rot}_{\mathrm{weak}}(F)$. To get (2), it is sufficient
first to prove that
$\mathrm{Cl}(\mathrm{Rot_{\mathrm{weak}}}(F))\subset\mathrm{Cl}(\mathrm{Rot}(F))$
in the case where $\mathrm{Rot}_{\mathrm{weak}}(F)$ is not reduced
to a point and then to explain why $\mathrm{Rot}_{\mathrm{weak}}(F)$
can not be reduced to $\{+\infty\}$ or $\{-\infty\}$.

Suppose that $\mathrm{Rot}_{\mathrm{weak}}(F)$ is not reduced to a
point and denote by $I$ the interior of its convex hull. According
to (1), we have
$$\mathrm{Cl}(\mathrm{Rot_{\mathrm{weak}}}(F))\subset\mathrm{Cl}(I)=\mathrm{Cl}
(I\cap\mathbb{Q})\subset\mathrm{Cl}(\mathrm{Rot}_{\mathrm{per}}(F))
\subset\mathrm{Cl}(\mathrm{Rot}(F)),$$ where
$\mathrm{Rot}_{\mathrm{per}}(F)$ is the set of rotation numbers of
periodic points.
\smallskip

Now we turn to explain why $\mathrm{Rot}_{\mathrm{weak}}(F)$ can not
be reduced to $\{+\infty\}$ or $\{-\infty\}$.

We state the first case, the other case is similar. Consider the
sphere $\mathbf{S}^2=\mathbb{A}\sqcup\{N,S\}$. Suppose first that
there exists a non trivial invariant continuum $K$ that contains the
end $N$ but not the end $S$. Consider the connected component $U$ of
$\mathbf{S}^2\setminus K$ that contains $S$. It is simply connected
(property of plane topology) and invariant by $f$. We can define the
\emph{prime end compactification} of $U$, introduced by
Carath\'{e}odory \cite{C}, by adding a circle $\mathbf{S}^1$. A very
good exposition of the theory of prime ends in modern terminology
can be found in the paper \cite{M} of Mather.

The prime end compactification can be defined purely topologically
but has another significance if we put a complex structure on
$\mathbf{S}^2$. We can find a conformal map $\phi$ between U and the
open disk $\mathbb{D}=\{z\in \mathbb{C}\mid |z|<1\}$ and we put on
$U\sqcup\mathbf{S}^1$ the topology (up to homeomorphism of the
resulting space, is independent of $\phi$) induced from the natural
topology of $\mathrm{Cl}(\mathbb{D})$ in $\mathbb{C}$ by the
bijection
$$\overline{\phi}:U\sqcup\mathbf{S}^1\rightarrow \mathrm{Cl}(\mathbb{D})$$
equal to $\phi$ on $U$ and to the identity on $\mathbf{S}^1$.

As $U$ is invariant by $f$, the map $f|_U$ can be extended to a
homeomorphism $\bar{f}$ of $U\sqcup\mathbf{S}^1$. Moreover,
$\bar{f}$ restricted to the prime ends set $\mathbf{S}^1$ is an
orientation preserving homeomorphism of the circle. Therefore, we
can define its rotation number, say $\rho\in
\mathbb{R}/\mathbb{Z}$.\smallskip

We take off $S$ from $U\sqcup\mathbf{S}^1$, and then paste the
annulus $\mathbf{S}^1\times[0,+\infty)$ and
$U\sqcup\mathbf{S}^1\setminus\{S\}$ together along $\mathbf{S}^1$.
We define now $f'$ as the extension of the map $\bar{f}$ to the new
open annulus $\mathbb{A}'$ as follows:
\begin{equation*}f'(x,y)=
\begin{cases}\bar{f}(x,y)& \textrm{if} \quad (x,y)\in U\sqcup\mathbf{S}^1\setminus\{S\};
\\(g(x),y)& \textrm{if} \quad
(x,y)\in\mathbf{S}^1\times[0,+\infty).\end{cases}
\end{equation*}
where $g$ is the map induced by $\bar{f}$ on $\mathbf{S}^1$. By the
construction of $f'$, the map $f'$ satisfies the intersection
property. Indeed, we suppose that $\gamma$ is an essential circle in
$\mathbb{A}'$. If $\gamma$ does not meet $S^1\times[0,+\infty)$, it
must meet its image because $f$ satisfies the intersection property.
Suppose now that $\gamma$ meets $S^1\times[0,+\infty)$. There exists
a real number $r\geq 0$ such that $\gamma $ meets $S^1\times\{r\}$
but does not meet $S^1\times(r,+\infty)$. This implies that
$f'(\gamma)$ and $f'^{-1}(\gamma)$ do not meet
$S^1\times(r,+\infty)$. As a consequence, one deduces that a point
$z\in S^1\times\{r\}$ cannot be strictly below the circles
$f'(\gamma)$ and $f'^{-1}(\gamma)$. This implies that $f'(\gamma)$
meets $\gamma$.\smallskip

For any lift $F'$ of $f'$ to the universal cover of $\mathbb{A}'$,
we have the following facts:
\begin{itemize}
 \item there exists $\rho\in\mathbb{R}$ such that
 $\mathrm{Rot}_{\mathrm{weak},\mathbf{S}^1\times\{r\}}(F')=\{\rho\}$ for every $r\geq 0$;
 \item $\mathrm{Rot}_{\mathrm{weak},K'}(F')=\{+\infty\}$ if the set $K'$ is a compact set of
 $U\setminus\{S\}$ that contains an essential circle.
\end{itemize}
So the set $\mathrm{Rot}_{\mathrm{weak}}(F')$ is not reduced to a
point. The assertion (1) implies that there are many periodic points
with a rotation number different from $\rho$. They correspond to
periodic points of our initial map $f$. We have a
contradiction.\smallskip

This implies that there exists a neighborhood of $N$ that is an
\emph{isolating Jordan domain} (see \cite{LY} for the details). Of
course, we have a similar situation for $S$. Observe that there are
no periodic orbits except $N$ and $S$ and that the intersection
property guarantees that $N$ and $S$ are neither sinks nor sources,
thanks to Le Calvez-Yoccoz's Theorem \cite{LY}, there exists an
integer $k$ such that the Lefschetz indices $i(f^k, S)$ and
$i(f^k,N)$ are non positive. But because of the Lefschetz formula,
we know that the sum is $2$ (the Euler characteristic of
$\mathbf{S}^2$). We have a contradiction.
\end{proof}

Observe that the proofs of Lemma \ref{lem:commute and topological
line} and Lemma \ref{lem:Fq satisfies the intersection property} do
not use the existence of recurrent points. To adapt the line
translation theorem with our new hypothesis, we only need to prove a
result similar to Lemma \ref{lem:bounded condition}.

\begin{lem}\label{lem:bounded condition weak}Let $f\in \mathrm{Homeo}_*^\wedge(\mathbb{A})$. We suppose that
$F$ is a lift of $f$ to $\mathbb{R}^2$ and that
$$\mathrm{Rot}_{\mathrm{weak}}(F)\subset [-q+2,q-2],$$
where $q\geq2$, then there exists an oriented essential line
$\gamma$ in $\mathbb{A}$ joining $S$ to $N$ that is lifted to an
essential line $\Gamma$ in $\mathbb{R}^2$ which satisfies
$$T^{-q}(\Gamma) < F(\Gamma) < T^{q}(\Gamma).$$
\end{lem}
\begin{proof} According to the assertion (2) of Proposition \ref{lem:gamma < f'gamma, then [0,+infty]}
 and Lemma \ref{lem:bounded condition}, we only
need to consider the case where $\mathrm{Rec}^+(f)=\emptyset$ and
$\mathrm{Rot}_{\mathrm{weak}}(F)$ is reduced to a real number
$\rho$.

By the assertion (4) of Proposition \ref{prop:properties of weak
rotation number}, we have $\mathrm{Rot}_{\mathrm{weak}}(F\circ
T^{q-1})=\{\rho+q-1\}\subset [1,2q-3]$. Of cause, the map $F\circ
T^{q-1}$ has no fixed point. We choose a compact set $K$ such that
$\mathrm{Rot}_{\mathrm{weak},K}(F\circ
T^{q-1})=\{\rho+q-1\}$. 
The map $f$ satisfying the intersection property, by Theorem
\ref{thm:GS} and Lemma \ref{lem:gamma < f'gamma, then [0,+infty]},
there is an oriented essential line $\gamma$ of $\mathbb{A}$ joining
$S$ to $N$ that is lifted to an essential Brouwer line $\Gamma$ of
$F\circ T^{q-1}$ which satisfies
$$\Gamma<(F\circ T^{q-1})(\Gamma).$$

Like in the proof of Lemma \ref{lem:bounded condition}, we deduce
that $F'=F\circ T^{q}$ is conjugate to a translation. Here again, we
consider the annulus $\mathbb{A}' = \mathbb{R}^2/F'$, the
homeomorphism $t$ induced by $T$ on $\mathbb{A}'$ and the lift
$F^{-1}\circ T^{q}= F'^{-1}\circ T^{2q}$ of $t^{\,2q}$. As it is
fixed point free, here again Theorem \ref{thm:GS} asserts that there
are three possible cases:

\begin{enumerate}
  \item There exists an oriented essential line $\gamma'$ of
$\mathbb{A}'$ that is lifted to an oriented line $\Gamma'$ in
$\mathbb{R}^2$ such that $\Gamma'<F'(\Gamma')$ and
$\Gamma'<(F^{-1}\circ T^{q})(\Gamma')$;
  \item There exists an oriented essential line $\gamma'$ of
$\mathbb{A}'$ that is lifted to an oriented line $\Gamma'$ in
$\mathbb{R}^2$ such that $\Gamma'<F'(\Gamma')$ and $(F^{-1}\circ
T^{q})(\Gamma')<\Gamma'$;
  \item There exists an essential circle in $\mathbb{A}'$ that is free
for $t^{\,2q}$.
\end{enumerate}
\bigskip

In the case (1), the arguments that we gave in the proof of Lemma
\ref{lem:bounded condition} permit us to construct an essential line
of $\mathbb{A}$ that satisfies the conclusion of the lemma.

%
%
%
%
\bigskip

In the case (2), we will show how to get a contradiction.

The line $\Gamma'$ satisfies $\Gamma'<(F\circ T^q)(\Gamma')$ and
$\Gamma'<(F\circ T^{-q})(\Gamma')$. Here again, the set $X$of couple
of integers $(m,n)$ such that $\Gamma'<(F^m\circ T^n)(\Gamma')$
contains the set of couples $(m, nq)$, where $m>0$ and $\vert n\vert
\leq m$ (see the proof of Lemma \ref{lem:bounded condition}). Let
$p$ be any given positive integer. For every couple $(m,n)$ of
integers such that $m>\frac{p}{q}$ and $\vert \frac{n}{m}\vert\leq
q-\frac{2p}{m}$, we have $qm-p>0$ and $\vert n-p\vert\leq\vert
n\vert+p\leq qm-p$. Therefore, we have $q(m,n)-p(1,q)\in X$, which
means
\begin{equation}\label{ineq:Gamma''}
    \Gamma'<(F\circ T^q)^p(\Gamma')< (F^m\circ T^n)^{q}(\Gamma').
\end{equation}

Let $\widetilde{K}_q=\pi^{-1}(\bigcup_{0\leq i<q}f^i(K))\cap
\mathrm{Cl}(R(T^{-q}(\Gamma))\cap L(\Gamma))$. Since $F'=F\circ T^q$
is conjugate to a translation, $\gamma'$ is an essential line in
$\mathbb{A}'$, we can always suppose that the compact set
$\widetilde{K}_q$ is contained in the region
$\mathrm{Cl}(R(\Gamma'))\cap L((F\circ T^q)^p(\Gamma'))$ for some
positive integer $p$ by replacing $\Gamma'$ with an iterate
$F'^k(\Gamma')$ if necessary. Consider the homeomorphism $f_q$ of
the annulus $\mathbb{A}_q=\mathbb{R}^2/T^q$ lifted by $F$ and write
$\pi_q:\mathbb{R}^2\rightarrow \mathbb{A}_q$ for the covering map.
Similarly to the proof of the assertion (4) of Proposition
\ref{prop:properties of weak rotation number}, we can find a
sequence $\{\widetilde{z}_k\}_{k\geq1}\subset\widetilde{K}_q$, a
sequence $\{m_k\}_{k\geq1}$ of positive integers and a sequence
$\{n_k\}_{k\geq1}$ of integers such that
\begin{eqnarray*}
   &\bullet& \lim_{k\rightarrow+\infty}m_k=+\infty;\qquad\qquad\qquad\qquad\qquad\qquad\qquad\qquad
   \qquad\qquad\qquad\qquad\qquad\\
   &\bullet& F^{qm_k}(\widetilde{z}_k)\in T^{-qn_k}(\widetilde{K}_q)\quad\mathrm{for\,\,\, every}\,\,\, k\geq1; \\
   &\bullet& \frac{-qn_k}{qm_k}\in ]-q+1,q-1[\quad\mathrm{when}\,\,\, k\,\,\,\mathrm{is\,\,\,
   large\,\,\, enough}.
\end{eqnarray*}

Therefore, there is a positive integer $N$ such that when $k\geq N$,
we have $m_k>2p>\frac{p}{q}$ and
$\vert\frac{n_k}{m_k}\vert<q-1<q-\frac{2p}{m_k}$.

By the inequation (\ref{ineq:Gamma''}), we have $\Gamma'<(F\circ
T^q)^p(\Gamma')< (F^{m_k}\circ T^{n_k})^{q}(\Gamma')$ when $k\geq
N$. On the one hand, the points of the sequence $\{(F^{m_k}\circ
T^{n_k})^{q}(\widetilde{z}_k)\}_{i\geq N}$ belong to $R((F\circ
T^q)^p(\Gamma'))$. On the other hand, the points $(F^{m_k}\circ
T^{n_k})^{q}(\widetilde{z}_k)$ belong to $\widetilde{K}_q\subset
L((F\circ T^q)^p(\Gamma'))$, which is a contradiction.
\bigskip

It remains to study the case (3) and try to find a contradiction.

Again, the arguments in the proof of Lemma \ref{lem:bounded
condition} permit us to construct a line $\Gamma'$ of $\mathbb{R}^2$
satisfying
\begin{itemize}
  \item $\Gamma'\cap T(\Gamma')=\emptyset$;
  \item $T^{-q}(\Gamma')=F(\Gamma')$.
\end{itemize}
Endow $\Gamma'$ with the orientation such that $\Gamma'<T(\Gamma')$.
As $\Gamma'$ is a Brouwer line of $T$, we can find an arc $\Delta_0$
that joins a point $z\in\Gamma'$ to $T(z)$ whose image is between
the two lines $\Gamma'$ and $T(\Gamma')$, and does not intersect the
lines but at the extremities. We deduce that $\Delta=\bigcup_{k\in
Z}T^k(\Delta_0)$ is the preimage of an essential circle $\delta$ of
$\mathbb{A}$. We know that
$\mathrm{Rot}_{\mathrm{weak},\delta}(F)=\{\rho\}$ (the intersection
property of $f$ guarantees that it is not empty). By (4) of
Proposition \ref{prop:properties of weak rotation number} and the
properties of $\Gamma'$, the map $F''=F\circ T^{q-1}$ satisfies the
following properties:
\begin{itemize}
\item $F''(\Gamma')=T^{-1}(\Gamma')<\Gamma'$;
\item $\Delta\subset \bigcup_{k\in\mathbb{Z}}T^k(\mathrm{Cl}(R(\Gamma))\cap
L(T(\Gamma)))$;
\item $\mathrm{Rot}_{\mathrm{weak},\delta}(F'')=\{\rho+q-1\}\subset
[1,2q-3]$.
\end{itemize}

By Lemma \ref{lem:gamma < f'gamma, then [0,+infty]}, the first two
items imply that $\mathrm{Rot}_{\mathrm{weak},\delta}(F'')\subset
[-\infty,0]$, which is contrary to the third item. We have completed
the proof.
\end{proof}

Replacing the notation $\mathrm{Cl}(\mathrm{Rot}(F))$ in the proof
of Theorem \ref{thm:the line translation theorem} by the notation
$\mathrm{Rot}_{\mathrm{weak}}(F)$, we have the following similar
theorem.

\begin{thm}[Generalization of the line translation theorem*] \label{thm:the line translation theorem*}
Let $f\in\mathrm{Homeo}_*^\wedge(\mathbb{A})$ and $F$ be a lift of
$f$ to $\mathbb{R}^2$. Assume that $\mathrm{Rot}_{\mathrm{weak}}(F)$
is contained in a Farey interval $]\frac{p}{q},\frac{p'}{q'}[$.
Then, there exists an essential line $\gamma$ in $\mathbb{A}$ such
that the lines $\gamma,f(\gamma),\cdots, f^{q+q'-1}(\gamma)$ are
pairwise disjoint. Moreover, the cyclic order of these lines is the
same as the cyclic order of the $q+q'-1$ first iterates of a
vertical line $\{\theta\}\times\mathbb{R}$ under the rigid rotation
with angle $\rho$, for any $\rho\in]\frac{p}{q},\frac{p'}{q'}[$.
\end{thm}\smallskip

\section{appendix}
\begin{lem}\label{subsec:positively
recurrent}Let $(X,d)$ be a metric space and $f: X\rightarrow X$ be a
continuous map. A positively recurrent point of $f$ is also a
positively recurrent point of $f^q$ for all $q\in
\mathbb{N}$.\end{lem}

\begin{proof}
If $z\in \mathrm{Rec}^+(f)$, let $O_i=\{z'\in X\mid
d(z,z')<\frac{1}{i}\}$ for $i\in \mathbb{N}\setminus\{0\}$. We
suppose that $f^{n_k}(z)\rightarrow z$ when $k\rightarrow +\infty$.
Write $n_k=l_kq+p_k$ where $0\leq p_k <q$. If there are infinitely
many $k$ such that $p_k=0$, we are done. Otherwise, there are
infinitely many $k$ such that $p_k=p$ where $0< p <q$. We can
suppose that $f^{l_kq+p}(z)\rightarrow z$ when $k\rightarrow
+\infty$ by considering subsequence if necessary. We suppose that
$f^{l_{k_1}q+p}(z)\in O_{m_1}$, then there exists $O_{m_2}$ such
that $f^{l_{k_1}q+p}(O_{m_2})\subset O_{m_1}$. Similarly, there
exists $l_{k_2}$ and $O_{m_3}$ such that
$f^{l_{k_1}q+p}(O_{m_3})\subset O_{m_2}$. By induction, there is a
subsequence $(l_{k_j})_{j\geq1}$ of $(l_{k})_{k\geq1}$ and a
subsequence $\{O_{m_j}\}_{j\geq1}$ of $\{O_m\}_{m\geq1}$ such that
$f^{l_{k_j}q+p}(O_{m_{j+1}})\subset O_{m_{j}}$. Consider the
subsequence $\{f^{q(p+\sum_{j=(t-1)q}^{tq-1}l_{k_j})}(z)\}_{t\geq
1}$, we are done.\end{proof}

\begin{lem}\label{lem:rec(fqq)}Let $f\in \mathrm{Homeo_*(\mathbb{A})}$ and $F$ be a lift
 of $f$ to $\mathbb{R}^2$. Define the homeomorphism
$f_q$ of the annulus $\mathbb{A}_q=\mathbb{R}^2/T^q$ lifted by $F$
and write $\pi_q:\mathbb{R}^2\rightarrow \mathbb{A}_q$ for the
covering map. If $z\in\mathrm{Rec}^+(f)$ and $\widetilde{z}\in
\pi^{-1}(z)$, then $\pi_q(\widetilde{z})\in
\mathrm{Rec}^+(f_q^q)$.\end{lem}

\begin{proof}To prove the lemma, we must find two
sequences $\{i_n\}_{n\geq 1}$ and $\{j_n\}_{n\geq 1}$ such that
$j_n\rightarrow+\infty$ and
\begin{equation}\label{rotation number for T'2q}
    T^{-qi_n}\circ
    F^{qj_n}(\widetilde{z})\rightarrow\widetilde{z}\quad(\mathrm{when}\quad
    n\rightarrow+\infty).
\end{equation}

Since $z\in \mathrm{Rec}^+(f^q)$, there exist two sequences of
integers $\{i_n\}_{n\geq 1}$ and $\{j_n\}_{n\geq 1}$  such that
$j_n\rightarrow+\infty$ and $T^{-i_n}\circ
F^{qj_n}(\widetilde{z})\rightarrow\widetilde{z}$ when $n$ tends to
$+\infty$. Let $\widetilde{O}_i=\{\widetilde{z}'\in \mathbb{R}^2\mid
d(\widetilde{z},\widetilde{z}')<\frac{1}{i}\}$ for
$i\in\mathbb{N}\setminus\{0\}$ where $d$ is the Euclidean metric of
$\mathbb{R}^2$. We suppose that $F^{qj_n}(\widetilde{z})\in
T^{i_n}(\widetilde{O}_{n})$ for every $n$ by considering subsequence
if necessary. Write $i_n=qs_n+t_n$ where $0\leq t_n <q$. If there
are infinitely many $n$ such that $t_n=0$, we are done. Otherwise,
there are infinitely many $n$ such that $t_n=p$ where $0< p <q$. We
can suppose that $F^{qj_n}(\widetilde{z})\in
T^{qs_n+p}(\widetilde{O}_{n})$ for every $n$ by further considering
subsequence if necessary.

We begin with $F^{qj_{n_1}}(\widetilde{z})\in T^{qs_{n_1}+p}(
\widetilde{O}_{n_1})$, then there exists $\widetilde{O}_{n_2}$ such
that $F^{qj_{n_1}}(\widetilde{O}_{n_2})\subset
T^{qs_{n_1}+p}(\widetilde{O}_{n_1})$. Similarly, there exists
$\widetilde{O}_{n_3}$ such that
$F^{qj_{n_2}}(\widetilde{O}_{n_3})\subset
T^{qs_{n_2}+p}(\widetilde{O}_{n_2})$. By induction, there is a
subsequence $\{j_{n_i}\}_{i\geq1}$ of $\{j_{n}\}_{n\geq1}$, a
subsequence $\{s_{n_i}\}_{i\geq1}$ of $\{s_{n}\}_{n\geq1}$ and a
subsequence $\{\widetilde{O}_{n_i}\}_{i\geq1}$ of $
\{\widetilde{O}_n\}_{n\geq1}$ such that
$F^{qj_{n_{i}}}(\widetilde{O}_{n_{i+1}})\subset
T^{qs_{n_i}+p}(\widetilde{O}_{n_i})$. Consider the subsequence
$\left\{T^{-q(p+\sum_{i=(t-1)q}^{tq-1}s_{n_i})}\circ
F^{q(\sum_{i=(t-1)q}^{tq-1}j_{n_i})}(\widetilde{z})\right\}_{t\geq
1}$, we are done.\end{proof}

\end{document}